\documentclass[noamsfonts]{amsart}
\usepackage[charter,expert]{mathdesign}
\usepackage[osf]{XCharter}
\selectfont
\usepackage[activate={true,nocompatibility},final]{microtype}

\usepackage[all, arc]{xy}
\SelectTips{eu}{}
\entrymodifiers={+!!<0pt,\fontdimen22\textfont2>}

\usepackage[pdftex, colorlinks=true, linkcolor=blue, citecolor=magenta, linktocpage]{hyperref}

\theoremstyle{plain}
\newtheorem{thm}[subsubsection]{Theorem}
\newtheorem{lemma}[subsubsection]{Lemma}
\newtheorem{prop}[subsubsection]{Proposition}
\newtheorem{cor}[subsection]{Corollary}

\theoremstyle{definition}
\newtheorem{example}[subsubsection]{Example}

\theoremstyle{definition}

\def\cF{\mathcal{F}}
\def\cG{\mathcal{G}}

\def\11{\mathbf{1}}

\def\CC{\mathbf{C}} 
\def\DD{\mathbf{D}}

\def\QQ{\mathbf{Q}} 
\def\RR{\mathbf{R}}

\def\fn{\mathfrak{n}}

\def\can{\mathrm{can}}

\def\DMmon{DM^b_{\mathrm{mon}}}

\def\id{\mathrm{id}}

\def\Re{\mathrm{Re}\,}
\def\rat{\mathrm{rat}}

\def\Spec{\mathrm{Spec}}

\def\var{\mathrm{var}}

\def\Dmon{D^b_{\mathrm{mon}}}

\newcommand{\mapright}[1]{\xrightarrow{#1}}

\makeatletter
\newcommand{\holim@}[2]{%
      \vtop{\m@th\ialign{##\cr
                  \hfil$#1\operator@font holim\,$\hfil\cr
                      \noalign{\nointerlineskip\kern1.5\ex@}#2\cr
                          \noalign{\nointerlineskip\kern-\ex@}\cr}}%
                      }
                      \newcommand{\holim}{%
                            \mathop{\mathpalette\holim@{\leftarrowfill@\textstyle}}\nmlimits@
                        }
                        \makeatother

                        \title{Fourier-Sato transform for monodromic mixed Hodge modules}
\author{R. Virk}
\begin{document}
\maketitle
\renewcommand{\thesubsection}{\textbf{\arabic{subsection}}}
\renewcommand{\thesubsubsection}{\textbf{\arabic{subsection}.\arabic{subsubsection}}}
\subsection{Introduction}
\subsubsection{}We discuss a simple formula for the Fourier-Sato transform of a monodromic mixed Hodge module on a complex algebraic variety.
Existing treatments that endow the Fourier-Sato transform with mixed Hodge structures involve constructions of explicit Hodge/weight filtrations on the underlying D-module/perverse sheaf. In contrast, we give a formal approach using the vanishing cycles functor.

\subsubsection{}No claims of originality are being made - most of the key points can be extracted from M. Kashiwara and P. Schapira's book \cite{KS} and J-L. Brylinski's work \cite{B}.

\subsection{Notation}\label{s:conventions}
\subsubsection{}The complex (resp. real) numbers are denoted $\CC$ (resp. $\RR$). We write $\Re\colon \CC \to \RR$ for the function that assigns to a complex number its real part. The multiplicative group of non-zero complex numbers is denoted $\CC^*$. 

\subsubsection{}`Canonical' and `natural' are synonyms for `natural transformation of functors'. The terms `map' and `morphism' are used interchangeably.

\subsubsection{}With the exception of statements involving mixed Hodge modules, sheaves are taken with coefficients in a commutative unital ring of finite global dimension. When working with mixed Hodge modules, the underlying (perverse) sheaves are taken with coefficients in the rational numbers $\QQ$.

\subsubsection{}
Functors on sheaves are always derived. I.e., we write $f_*$ instead of $Rf_*$, etc.
Given a locally closed subspace $v\colon Z \to X$ and a complex of sheaves $F$, we set:
\[ \Gamma_Z(F) = v_*v^!(F). \]

\subsubsection{}A variety will mean a separated scheme of finite type over $\Spec(\CC)$. 
The bounded derived category of algebraically constructible sheaves, on the complex analytic space associated to a variety $X$, will be denoted $D^b_c(X)$. Given an algebraic action of $\CC^*$ on $X$, we write $\Dmon(X) \subset D^b_c(X)$ for the full subcategory of monodromic objects. I.e., those $F\in D^b_c(X)$ such that, for all $i$, the cohomology sheaves $H^i(F)$ are local systems when restricted to any $\CC^*$-orbit. Note: monodromic objects are not necessarily $\CC^*$-equivariant.

\begin{example}Consider $\CC^*$ acting on itself via multiplication. Then every local system on $\CC^*$ is monodromic. However, only trivial local systems are $\CC^*$-equivariant.
\end{example}

\subsubsection{}If $V\to S$ is an algebraic vector bundle, unless explicitly stated otherwise, $V$ will always be equipped with the $\CC^*$-action that acts via dilating the fibres of $V\to S$.

\pagebreak
\subsection{Preliminaries}\label{s:prelims}

\subsubsection{}
The following variation on the homotopy invariance of cohomology is a key tool. It is well known in the context of equivariant sheaves. The extension to monodromic sheaves requires no new ideas.
\begin{lemma}\label{homotopylemma}
Let $S$ be a variety and $\tau\colon V\to S$ an algebraic vector bundle. Write $i\colon S\to V$ for the zero section.
Let $F\in \Dmon(V)$. Then $i^*$ applied to the adjunction map $\tau^*\tau_* F \to F$ yields a canonical isomorphism $\tau_*F \mapright{\cong}i^*F$.
\end{lemma}

\begin{proof}There are several ways to extract this from the literature. Let me give three:
\begin{enumerate}
\item By d\'evissage, it suffices to treat the case $i^*F = 0$. This is \cite[Lemme 6.1]{V}.
\item Monodromic sheaves are conic in the sense of \cite[\S3.7]{KS}. So our assertion is simply a special case of \cite[Proposition 3.7.5]{KS}.
\item 
Replace $\CC^*$ in \cite[Proposition 1]{So} with the subgroup consisting of strictly positive real numbers. The (formal) argument goes through verbatim. Note that this yields a slightly more general result for contracting actions (i.e., not just vector bundles). This argument goes back (at least) to \cite{Sp}. \qedhere
\end{enumerate}
\end{proof}

\begin{example}Consider $\CC$ as a vector bundle over a point. Write $j\colon \CC - \{0\} \hookrightarrow \CC$ for the inclusion. 
Let $L$ be a local system on $\CC - \{0\}$. Then $j_*L$ is monodromic. Lemma \ref{homotopylemma} recovers the familiar fact that the stalk cohomology of $j_*L$ at $0$ is isomorphic to $H^*(\CC - \{0\}; L)$.
\end{example}

\subsubsection{}Although not strictly needed for our purposes, this seems to be a convenient opportunity to clarify the connection between `conic sheaves' and `monodromic sheaves' in the algebraic context.

Let $\RR^+$ denote the multiplicative group of strictly positive real numbers. Suppose $X$ is a topological space endowed with an $\RR^+$-action. Then a sheaf on $X$ is called \emph{conic} if its restriction to each $\RR^+$-orbit is a local system.

Now suppose $X$ is a variety equipped with an algebraic $\CC^*$-action. Along with monodromic sheaves, we may also consider consider conic sheaves for the restricted action of $\RR^+\subset \CC^*$. Certainly, a monodromic sheaf is conic. The converse is also true for algebraically constructible sheaves.
\begin{prop}\label{conicmonodromic}An algebraically constructible conic sheaf is monodromic.
\end{prop}

\begin{proof}
The assertion reduces to showing that if an algebraically constructible sheaf $F$ on $\CC^*$ is conic, then it is a local system. For such an $F$, algebraic constructibility implies $F$ is a local system away from some finite number of points. For each such point, pick a local coordinate and consider the vanishing cycles of $F$ with respect to the coordinate. As $F$ is conic, these vanishing cycles are zero. Hence, $F$ is a local system over said point.
\end{proof}

\subsubsection{}
Vanishing cycles will continue to appear(!) below.
Given a morphism of varieties $f\colon X\to \CC$, we write $\phi_f\colon D^b_c(X) \to D^b_c(f^{-1}(0))$ for the vanishing cycles functor associated to $f$. Shift conventions are those of \cite[Chapter VIII]{KS}. That is, our $\phi_f$ is exact for perverse sheaves.
We have a canonical isomorphism (see \cite[Exercise VIII.13]{KS}):
\begin{equation}\label{topvanish} \phi_f(F) \mapright{\cong} \Gamma_{\{\Re f \geq 0\}}(F)|_{f^{-1}(0)}.\end{equation}

\pagebreak
\subsection{Fourier-Sato transform}\label{s:sato}
\subsubsection{}
Let $S$ be a variety and let $V\to S$ be an algebraic vector bundle. Write $V^{\vee}\to S$ for the dual bundle and let $\mu\colon V\times_S V^{\vee}\to \CC$ denote the evaluation pairing. Then we have the diagram:
\[ \xymatrix{
& V\times_S V^{\vee}\ar[ld]_{p}\ar[rd]^q \ar[rr]^{\mu} & & \CC \\
V & & V^{\vee}
}\]
where $p$ and $q$ are the evident projections. 
Let $F\in \Dmon(V)$. By definition \cite[Definition 3.7.8]{KS}, the \emph{Fourier-Sato transform} $F^{\wedge}$, is given by :
\[ F^{\wedge} =  q_*\Gamma_{\{\Re \mu\geq 0\}}(p^*F). \]

\subsubsection{}
Consider the diagram:
\[
\xymatrix{
& V\times_s V^{\vee} \ar[ld]_p \ar[rd]^q \ar[rr]^{\gamma} & & V^{\vee} \times \CC \ar[ld]_{\tau}\ar[r]^-t & \CC \ar[r]^{\Re} & \RR \\
 V & & V^{\vee}
}
\]
where $t$ and $\tau$ are the evident projections, and $\gamma(v,w) = (w,\mu(v,w))$.
The following is essentially contained in the proof of \cite[Proposition 10.3.18]{KS}.
\begin{prop}\label{firstmain}Let $F\in \Dmon(V)$. Then there is a canonical isomorphism:
\[ F^{\wedge} \mapright{\sim} \phi_t\gamma_*p^*(F). \]
\end{prop}

\begin{proof}
Write $i\colon V^{\vee} \to V^{\vee}\times \CC$ for the zero section. Then:
\begin{align*}
F^{\wedge} &= q_*\Gamma_{\{\Re \mu\geq 0\}}(p^*F) &\text{(by definition)} \\
&\cong \tau_*\gamma_*\Gamma_{\{\Re\mu\geq 0\}}(p^*F) \\
&\cong \tau_*\Gamma_{\{\Re t \geq 0\}}(\gamma_*p^*F) \\
&\cong i^*\Gamma_{\{\Re t \geq 0\}}(\gamma_*p^*F) &\text{(Lemma \ref{homotopylemma})} \\
&\cong \phi_t\gamma_*p^*(F) & \text{(by \eqref{topvanish}).} 
\end{align*}
\end{proof}

\begin{cor}The Fourier-Sato transform defines a functor $\Dmon(V)\to \Dmon(V^{\vee})$.
\end{cor}

\begin{proof}Algebraic constructibility follows from that of the vanishing cycles functor. That the Fourier-Sato transform of a monodromic sheaf is in turn monodromic can be seen directly. Alternatively, apply Proposition \ref{conicmonodromic} (the Fourier-Sato transform of a conic sheaf is in turn conic).
\end{proof}

\begin{example}\label{egquiver}
Work with sheaf coefficients in a field. Write $PerMon(\CC)$ for the category of monodromic perverse sheaves on $\CC$ (viewed as the trivial vector bundle over a point). It will be convenient to work with the shifted\footnote{To ensure t-exactness for perverse sheaves.} Fourier-Sato transform: for $F\in PerMon(\CC)$, set $\cF(F) = (F^{\wedge})[1]$.
Let $Glue$ be the category of diagrams: 
\[ \Psi \mapright{\can} \Phi \mapright{\var} \Psi \]
where $\Psi$ and $\Phi$ are finite dimensional vector spaces and $\id - \var\circ \can$ is an automorphism (equivalently $\id-\can\circ\var$ is an automorphism).
Then we have an equivalence of categories: 
\[ \cG \colon PerMon(\CC) \mapright{\cong} Glue \]
that sends a monodromic perverse sheaf $F$ to the diagram
\[ \psi(F)[-1] \mapright{\can} \phi(F) \mapright{\var} \psi(F)[-1] \]
Here $\psi$ and $\phi$ denote the nearby and vanishing cycles functor\footnote{Shift conventions are those of \cite[\S VIII.6]{KS}. So $\psi[-1]$ is t-exact for perverse sheaves.}, respectively, for the identity map on $\CC$. The maps $\can, \var$ are the canonical and variation morphisms, respectively (see \cite[\S VIII.6]{KS}). Under this equivalence, $\cF$ corresponds to swapping $\Psi$ and $\Phi$ (see \cite[Th\'eor\'eme 3.1]{BMV}). I.e., $\cG(\cF(F))$ is canonically isomorphic to the diagram:
\[ \phi(F) \mapright{\var} \psi(F)[-1] \mapright{\can} \phi(F) \]
\end{example}

\subsubsection{}
For $F\in \Dmon(V)$ one knows that there is \emph{some} canonical isomorphism: 
\[ \phi_t\gamma_!p^*(F) \mapright{\cong}\phi_t\gamma_*p^*(F) \]
 (use \cite[Theorem 3.7.7]{KS} and the argument of Proposition \ref{firstmain}). We need the slightly more precise:
\begin{prop}\label{secondmain}Let $F\in \Dmon(V)$. Then the canonical `forget supports' map $\gamma_! \to \gamma_*$ yields an isomorphism:
\[ \phi_t\gamma_!p^*(F) \mapright{\cong} \phi_t\gamma_*p^*(F).\]
\end{prop}

\begin{proof}
We may assume $F$ is perverse. Let $C$ denote the cone (in the sense of triangulated categories) of the map $\gamma_!p^*(F) \to \gamma_*p^*(F)$. According to \cite[Proposition 4.8]{B} the restriction of $C$ to any $\{y\} \times \CC$, with $y\in V^{\vee}$, is a local system (up to shift). Consequently, $\phi_t(C) = 0$.\footnote{An alternative argument that uses the fact that the transform is an equivalence and exact (up to shift) for perverse sheaves is to observe that it is sufficient to show that the induced map is non-zero on an irreducible perverse sheaf. Verifying this requires a diagram chase.}
\end{proof}

\subsection{Mixed Hodge modules}\label{mhm}
\subsubsection{}
For a variety $X$, write $DM^b(X)$ for the bounded derived category of mixed Hodge modules on $X$. By construction, $DM^b(X)$ is equipped with a conservative t-exact (for the perverse t-structure) functor:
\[ \rat\colon DM^b(X) \to D^b_c(X).\]
In particular, given a morphism $\alpha$ in $DM^b(X)$, if $\rat(\alpha)$ is an isomorphism, then so is $\alpha$. The point is that while it can be difficult to construct a map of mixed Hodge modules, verifying that one is an isomorphism is a problem that can be tackled in $D^b_c(X)$. All the usual functors/adjunctions lift to $DM^b(X)$ compatibly with $\rat$.

\subsubsection{}
If $X$ is endowed with an algebraic $\CC^*$-action, then we let
$\DMmon(X)\subset DM^b(X)$ denote the full subcategory consisting of $M\in DM^b(X)$ such that $\rat(M) \in \Dmon(X)$.
 
\subsubsection{}With the notation of \S\ref{s:sato}, for $M\in \DMmon(V)$, \emph{define}:
\[ M^{\wedge} = \phi_t\gamma_*p^*(M).\]
\begin{thm}\label{mixedfourier}The Fourier-Sato transform lifts to mixed Hodge modules. More precisely, we have a canonical isomorphism:
\[ \rat(M^{\wedge}) \cong \rat(M)^{\wedge}.\]
\end{thm}

\begin{proof}Immediate from Proposition \ref{firstmain}.
\end{proof}

\begin{thm}Let $V\to S$ be an algebraic vector bundle of constant rank $r$. Write $\DD$ for Verdier duality, and $(r)$ for the $r$-th Tate twist on mixed Hodge modules. Let $M\in \DMmon(V)$. Then we have a canonical isomorphism:
\[ \DD(M^{\wedge}) \cong (\DD M)^{\wedge}[2r](r). \]
\end{thm}

\begin{proof}We have:
\begin{align*}
\DD(M^{\wedge}) &= \DD(\phi_t\gamma_*p^*(M)) & \text{(by definition)} \\
&\cong \phi_t\gamma_!p^!(\DD M) \\
&\cong \phi_t\gamma_!p^*(\DD M)[2r](r) & \text{($p$ is smooth)}\\
&\cong \phi_t\gamma_*p^*(\DD M)[2r](r) & \text{(Proposition \ref{secondmain})} \\
&= (\DD M)^{\wedge}[2r](r). 
\end{align*}
\end{proof}

\subsection{Complements on Fourier inversion}
\subsubsection{}
It seems plausible that $(-)^{\wedge}$ yields an equivalence of categories on mixed Hodge modules.
I do not know a \emph{simple} demonstration of this: it suffices to show that the adjoint to the transform (on the ordinary derived category of sheaves) lifts to Hodge modules. However, the presence of $\phi_t$ in the formula of Proposition \ref{firstmain} makes seeing that an adjunction (for Hodge modules) exists a bit murky.
Here is a sketch of an indirect argument.\footnote{I would be very grateful to anyone who would explain to me a more satisfying way of seeing this.}

Use the gluing description for $\Dmon(X\times \CC)$ given by \cite[Proposition 0.3]{Sa} to obtain a quiver description as in Example \ref{egquiver} (see \cite[\S3]{BMV} for the general version). Conclude that (ignoring Tate twists) the Fourier-Sato transform acts as an involution on the quiver - an evident equivalence. Now using the functoriality of the Fourier-Sato transform under maps of vector bundles (which we have not proved for mixed Hodge modules in this note) one sees that the general transform can be expressed in terms of the transform on the trivial bundle and standard pushforward/pullback functors that admit adjoints (the formula under \cite[Corollaire 7.20]{B}). This shows that the transform admits an adjoint on monodromic mixed Hodge modules. The unit and counit of adjunction must be isomorphisms because they are so on applying $\rat$.

\begin{example}
Mimicking Example \ref{egquiver}, write $MHMmon(\CC)$ for the category of monodromic mixed Hodge modules on $\CC$. Fix an orientation of $\CC$ so that we may identify a $\QQ$-vector space with its Tate twists.

Again, it will be convenient to work with the shifted Fourier-Sato transform: for $F\in MHMmon(\CC)$, set $\cF(F) = (F^{\wedge})[1]$.
Let $Glue$ be the category of diagrams:
\[ \Psi \mapright{\can} \Phi \mapright{\var} \Psi(-1) \]
where $\can$ and $\var$ are morphisms of polarizable mixed Hodge structures such that on the $\QQ$-vector space underlying $\Psi$, the induced linear map $\id - \var\circ \can$ is invertible.
Then we have an equivalence of categories: 
\[ \cG \colon MHMmon(\CC) \mapright{\cong} Glue \]
that sends a monodromic mixed Hodge module $M$ to the diagram:
\[ \psi(F)[-1] \mapright{\can} \phi(F) \mapright{\var} \psi(F)[-1](-1) \]
The notation is as in Example \ref{egquiver}. Under this equivalence, $\cF$ \emph{should}\footnote{I have not checked the details.} correspond to swapping $\Psi$ and $\Phi$. I.e., $\cG(\cF(F))$ should be canonically isomorphic to the diagram:
\[ \phi(F) \mapright{\var} \psi(F)[-1](-1) \mapright{\can} \phi(F)(-1) \]
\end{example}

\end{document}